\documentclass[12pt,a4paper]{article}
\usepackage[utf8]{inputenc}
\usepackage[T1]{fontenc}
\usepackage[slovak,english]{babel}
\usepackage{amsthm}
\usepackage{amssymb}
\usepackage{amsbsy}
\usepackage{amsmath}
\usepackage{pgf,tikz}
\usepackage{asymptote}
\usepackage{enumerate}
\usetikzlibrary{arrows}

\newtheorem{theorem}{Theorem}[section]
\newtheorem{lemma}[theorem]{Lemma}
\newtheorem{corollary}[theorem]{Corollary}
\theoremstyle{definition}
\newtheorem{remark}[theorem]{Remark}
\newtheorem{definition}[theorem]{Definition}
\newtheorem*{example}{Example}
\usepackage{url}

\begin{document}

\begin{center}

\Large
\textbf{An upper bound of a
generalized upper Hamiltonian number of a graph}

\end{center}

\begin{center}
\large
Martin Dz\'urik
\end{center}

\begin{center}
\textit{Department of Mathematics and statistics, Faculty of Science,
Masaryk University, Kotl\'a\v rsk\'a 2, CZ-61137 Brno, Czech Republic}
\end{center}

\begin{center}

\end{center}
\begin{center}

\end{center}

\begin{center}
\textbf{Abstract}
\end{center}

In this article we study graphs with ordering of vertices, 
we define a generalization called a pseudoordering, and for a graph $H$ we define the  
$H$-Hamiltonian number of a graph $G$. We will show that this concept 
is a generalization of both the Hamiltonian number and the traceable number. 
We will prove equivalent characteristics of an isomorphism of graphs $G$ and $H$ using
$H$-Hamiltonian number of $G$.
Furthermore, we will show that 
for a fixed number of vertices, each path has a maximal upper $H$-Hamiltonian 
number, which is a generalization of the same claim for upper Hamiltonian 
numbers and upper traceable numbers. 
Finally we will show that for every connected graph $H$ 
only paths have maximal $H$-Hamiltonian number.

\newpage

\section{Introduction}

In this article we study a part of graph theory based on an ordering of vertices. We define a generalization called a pseudoordering of a graph.
We will show how to generalize a Hamiltonian number, for a graph $H$ we define the  
$H$-Hamiltonian number of a graph $G$ and we will show that this concept 
is a generalization of both the Hamiltonian number and the traceable number. We get them by a special choice of graph $H$. 
Furthermore, we will study a maximalization of upper $H$-Hamiltonian number for a fixed number of vertices. We will show that,
for a fixed number of vertices, each path has a maximal upper $H$-Hamiltonian 
number. From the definition it will be obvious that a lower bound of the $H$-Hamiltonian number is the number of edges $|E(H)|$ and the graph $G$ has a minimal lower $H$-Hamiltonian number if and only if $H$ is a subgraph of $G$. Now we can say that $G$ having a maximal upper $H$-Hamiltonian number is dual to $H$ being a subgraph of $G$. Furthermore, by above for every two finite graphs $G$ and $H$ such that $G$ is connected satisfying $|V(G)|=|V(H)|$ and 
$|E(G)|=|E(H)|$, we get that $G\cong H$ if and only if the lower $H$-Hamiltonian number of $G$ is $|E(H)|$.

In \cite{2} it is proved that $G$ has a maximal upper traceable number if and only if $G$ is a path. The same is proved for Hamiltonian number. We will show that for $H$ connected $G$ has a maximal $H$-Hamiltonian number if and only if $G$ is a path.
This shows that this generalization of ordering of vertices is natural. 

This aricle is based on the bachelor thesis
\cite{bak}. The author would like to thank Jiří Rosický for many helpful discussions.
\begin{center}

\end{center}

In this article we will study a generalization of Hamiltonian { spectra} of
undirected finite graphs.
Recall that,
a graph $G$ is {a} pair 
\[
G=(V(G),E(G)),
\]
where $V(G)$ is {a} finite set of vertices of $G$ and $E(G)\subseteq V(G) \times V(G)$,
a {symmetric} antireflexive relation, {is a set of edges}. {We will denote an edge between $v$ and $u$ by $\{v,u\}$.}

Recall that,
an {\it ordering} on {the} graph $G$ is a bijection $f\,:\,\{1,2, \dots, |V(G)|\}\rightarrow V(G)$, we denote

\[
 s(f,G)=\sum_{i=1}^{|V(G)|} \rho_G (f(i),f(i+1)), 
\]
\[
 \bar{s}(f,G)=\sum_{i=1}^{|V(G)|-1} \rho_G (f(i),f(i+1)), 
\]  where
$\rho_G (x,y)$ {is the} distance of $x,y$ in {the} graph $G$ and
$f(|V(G)|+1):=f(1)$,
for better notation. We will write only $s(f)$, $\bar{s}(f)$ if {the} graph is clear from 
{context}. Then 

\[
\{s(f,G)|f~\text{ordering on}~ G \}
\]
\[
\{\bar{s}(f,G)|f~\text{ordering on}~ G \}
\]
{are} {the} {\it Hamiltonian spectrum} of {the} graph $G$ and {the} {\it traceable spectrum} of {the} graph $G$, {respectively}.

We want to {generalize} {the notion of an ordering} of {a} graph.

\begin{definition}
Let $G,H$ {be} graphs such that $|V(G)|=|V(H)|$ and \\ $f\,:\,V(H)\rightarrow V(G)$ {is a}
bijection, then we call $f$ {a} {\it pseudoordering} on {the} graph $G$ (by $H$), 
denote 
\[ 
s_H(f,G)=\sum_{\{x,y\}\in E(H)}\rho_G(f(x),f(y)),
\] 
where
$\rho_G (x,y)$ is {the} distance of $x,y$ in {the} graph $G$. 
We will call $s_H(f,G)$ {the} sum of {the} pseudoordering $f$. Then 

\[
\{s_H(f,G)|f~\text{pseudoordering on}~ G \text{ by $H$}\}
\]
is {the} {\it $H$-Hamiltonian spectrum} of {the} graph $G$.

\end{definition}

{The} minimum and {the} maximum of {a} Hamiltonian spectrum and {of a} traceable spectrum are
called {the} {\it (lower) Hamiltonian number} {and the} {\it upper Hamiltonian number}, {respectively}. {Furthermore, the} (lower)
traceable number and {the} upper traceable number of {a} graph $G$ {are} denoted by

\begin{equation}
\begin{aligned}
h(G)&=\min\{s(f,G)|f ~\text{ordering on}~ G \},\\
h^+(G)&=\max\{s(f,G)|f ~\text{ordering on}~ G \},\\
t(G)&=\min\{\bar{s}(f,G)|f ~\text{ordering on}~ G \},\\
t^+(G)&=\max\{\bar{s}(f,G)|f ~\text{ordering on}~ G \}.
\end{aligned}\notag
\end{equation}
Now we {define generalized} versions.

\begin{definition}
\begin{equation}
\begin{aligned}
h_{H}(G)&=\min\{s_H(f,G)|f ~\text{pseudoordering on}~ G \},\\
h_{H}^+(G)&=\max\{s_H(f,G)|f ~\text{pseudoordering on} ~ G \}.
\end{aligned}\notag
\end{equation}
We will call them {the} {\it lower $H$-Hamiltonian number} and {the} {\it upper $H$-Hamiltonian number} of {a} graph $G$, {respectively}.
\end{definition}

Now take $H=C_{|V(G)|}$, where $C_n$ is {the} cycle with $n$ {vertices}. When we 
denote {the} vertices of $C_{|V(G)|}$ by $\{1,2,\dots,|V(G)|\} $
we can see {that}
\[
s(f,G)=s_{C_{|V(G)|}}(f,G).
\]

{Analogously} for $H=P_{|V(G)|-1}$, where $P_{n-1}$ is {the} path of {length} $n-1$, {we get that}

\[
\bar{s}(f,G)=s_{P_{|V(G)|-1}}(f,G).
\]

\begin{remark}
{The} $C_{|V(G)|}$-Hamiltonian spectrum of {a graph} $G$ is equal to {the} Hamiltonian spectrum of
$G$ for $|V(G)|\geq 3 $, and {the} $ P_{|V(G)|-1}$-Hamiltonian spectrum of $G$ is equal to {the} traceable
spectrum of $G$ for $|V(G)|\geq 2 $.
\end{remark}

\begin{lemma}
Let $G$ be a connected finite graph and $H$ {be} a graph such that
$|V(G)|=|V(H)|$, {then} $h_H(G)= |E(H)|$ if and only if $H$ is isomorphic to some subgraph of $G$.
\end{lemma}

\begin{proof}
Let $f\,:\,V(H)\rightarrow V(G)$ be a pseudoordering satisfying 
$s(f,G)= |E(H)|$, then $f$ 
is an injective graph homomorphism. 
The opposite implication is obvious.
\end{proof}

\begin{lemma}
Let $G$ be a connected finite graph and $H$ {be} a graph such that
$|V(G)|=|V(H)|$ and $|E(G)|=|E(H)|$, {then} $h_H(G)= |E(H)|$ if and only if $H$ is isomorphic to the graph $G$.
\end{lemma}

\begin{proof}
The graph $H$ is isomorphic to a subgraph of $G$, and furthermore $|V(G)|=|V(H)|$, $|E(G)|=|E(H)|$, hence $H\cong G$.
The opposite implication is obvious.
\end{proof}

\section{Maximalization of the upper $H$-Hamiltonian number of a graph $G$}

In this section we will prove {that} for every {pair of} connected graphs $H,G$ and {each} pseudoordering $f$ 
{there} exists {a} pseudoordering \\ $g:V(H)\rightarrow \{1,2,\dots,
|V(G)|\}$ {such that} $s_H(f,G) \leq s_H(g,P_{|V(G)|-1})$.
At first, let $G$ be a tree. We will only work with graphs which have al least 2 vertices.

\begin{definition}

Let $G$ and $H$ be graphs such that $G$ is connected, \\ $|V(G)|~=~|V(H)|$ and 
$f:V(H)\rightarrow V(G)$ {is a} pseudoordering. {Furthermore}, let $a,b \in V(G)$, we define
$a\sim_{H,f} b$ if and only if $\{f^{-1}(a),f^{-1}(b)\} \in E(H)$.
\end{definition}

\begin{definition}
{Let} $G$ be a tree {such that} $G$ is not a path. Denote
three {pairwise distinct leaves by} $l,k,v\in
V(G)$. Because $G$ is not a path than $G$ has at least $3$ leaves,
connect $l,k$ with a path $l=x_1,x_2,\dots, x_m=k$. {Connect $v,l$ with a path} $l$ $\quad v=y_1,y_2,\dots ,y_s=l$ and take {the} minimum of a set
\[
i_{m}=\min\{i|\exists j \in \{1,\dots ,m \},y_i=x_j\}.
\]
Take $j_{m}$ {such that} $y_{i_m}=x_{j_m}$.
Now we define $ u=y_{i_m}$, $w=y_{i_m-1}$, $u^{+}=x_{j_m-1}$, $u^{-}=x_{j_m+1}$.
\end{definition}

\begin{example}

\begin{flushleft}
\end{flushleft}

\begin{center}
\definecolor{qqqqff}{rgb}{0,0,1}
\definecolor{xdxdff}{rgb}{0.49,0.49,1}
\begin{tikzpicture}[line cap=round,line join=round,>=triangle 45,x=1.0cm,y=1.0cm]
\clip(-3.57,-2.22) rectangle (3.33,3.38);
\draw (0,3)-- (0,2);
\draw (0,2)-- (0,1);
\draw (0,1)-- (0,0);
\draw (0,0)-- (0,-1);
\draw (0,-1)-- (0,-2);
\draw (0,0)-- (-1,0);
\draw (-1,0)-- (-2,0);
\draw (-2,0)-- (-3,0);
\draw (-2,0)-- (-2.32,0.48);
\draw (-2.32,0.48)-- (-2.72,0.76);
\draw (-1,0)-- (-1.42,0.72);
\draw (-1.42,0.72)-- (-1.88,1.16);
\draw (0,1)-- (0.84,1.28);
\draw (0.84,1.28)-- (1.56,1.2);
\draw (0,-1)-- (0.88,-0.86);
\draw (0.88,-0.86)-- (1.82,-0.56);
\draw (0,0)-- (0.66,0.4);
\draw (0.66,0.4)-- (1.58,0.44);
\draw (0.84,1.28)-- (1.14,1.78);
\draw (0,2)-- (0.82,2.44);
\draw (1.58,0.44)-- (2.48,0.44);
\draw (-2,0)-- (-2.14,-0.8);
\draw (0.07,1.7) node[anchor=north west] {$ u^+ $};
\draw (0.09,-0.3) node[anchor=north west] {$u^-$};
\draw (0.05,0.7) node[anchor=north west] {$u$};
\begin{scriptsize}
\fill [color=black] (0,3) circle (1.5pt);
\draw[color=black] (0.06,3.15) node {$l$};
\fill [color=xdxdff] (0,2) circle (1.5pt);
\fill [color=black] (0,1) circle (1.5pt);
\fill [color=black] (0,0) circle (1.5pt);
\fill [color=black] (0,-2) circle (1.5pt);
\draw[color=black] (0.09,-1.85) node {$k$};
\fill [color=black] (-1,0) circle (1.5pt);
\draw[color=black] (-0.9,0.15) node {$w$};
\fill [color=xdxdff] (-2,0) circle (1.5pt);
\fill [color=black] (-3,0) circle (1.5pt);
\draw[color=black] (-2.92,0.15) node {$v$};
\fill [color=qqqqff] (-2.32,0.48) circle (1.5pt);
\fill [color=qqqqff] (-2.72,0.76) circle (1.5pt);
\fill [color=qqqqff] (-1.42,0.72) circle (1.5pt);
\fill [color=qqqqff] (-1.88,1.16) circle (1.5pt);
\fill [color=qqqqff] (0.84,1.28) circle (1.5pt);
\fill [color=qqqqff] (1.56,1.2) circle (1.5pt);
\fill [color=qqqqff] (0.88,-0.86) circle (1.5pt);
\fill [color=qqqqff] (1.82,-0.56) circle (1.5pt);
\fill [color=qqqqff] (0.66,0.4) circle (1.5pt);
\fill [color=qqqqff] (1.58,0.44) circle (1.5pt);
\fill [color=qqqqff] (1.14,1.78) circle (1.5pt);
\fill [color=qqqqff] (0.82,2.44) circle (1.5pt);
\fill [color=qqqqff] (2.48,0.44) circle (1.5pt);
\fill [color=qqqqff] (-2.14,-0.8) circle (1.5pt);
\fill [color=black] (0,-1) circle (1.5pt);
\end{scriptsize}
\end{tikzpicture}
\end{center}

\end{example}

\begin{remark}
$l\neq u \neq k$.
\end{remark}

\begin{definition}

Define a set $K(v,G) \subseteq V(G)$ as a set of vertices $z \in V(G)$ 
such the path between $z$ and $l $ {uses the} edge $\{w,u\}$. 
\end{definition}

\begin{remark}
$K(v,G)$ is the connected component of
$(V(G),E(G)\setminus \{w,u\})$, $G$ without edge $\{w,u\}$, which contains $v$.
\end{remark}

\begin{lemma}\label{l8}
\begin{enumerate}[(i)]
\item Paths between vertices from $K(v,G)$ don't use {the} edge $\{w,u\}$.
\item Paths between vertices from $V(G) \setminus K(v,G)$ don't use {the} edge
$\{w,u\}$.
\item Paths {joining a vertex} from $V(G) \setminus K(v,G)$ {to a vertex from} $K(v,G)$ use {the}
edge $\{w,u\}$.

\end{enumerate}
\end{lemma}

\begin{proof}

\begin{enumerate}[(i)]

\item If $a,b\in K(v,G),~ a\neq b$ connect them with a path with $l $ \\ $ l=a_1,a_2,\dots ,a_p=a $
and
$l=b_1,b_2,\dots ,b_q=b$. We know {that} $a_i\neq b_j$ {for} $i \neq j$, {a} path
between two vertices in {a} tree is {uniquely determined}.
Consider the last common {vertex} $\tilde{a}=a_{max\{i|a_i=b_i\}}$. We know {that}
paths between $a$, $l$ and $b$, $l$
use {the} edge $\{w,u\}$. Thus they still coincide on it, {therefore} paths between $\tilde{a}$,$a$
and $\tilde{a}$,$b$
don't use $\{w,u\}$ and except for $\tilde{a}$ are disjoint.
We can join them and we get {a} path between $a$ and $ b $ which doesn't use $\{w,u\}$.

\item Let $a,b\in V(G)\setminus K(v,G)\quad a\neq b${,} analogically we take
paths to $l$,\\ $ l=a_1,a_2,\dots ,a_p=a $ and
$l=b_1,b_2,\dots ,b_q=b$,
$\tilde{a}=a_{max\{i|a_i=b_i\}}$. Paths {between} $a$,$l$ and $b$,$l$
don't use edge $\{w,u\}$, then paths $\tilde{a}$,$a$ and $\tilde{a}$,$b$
don't use $\{w,u\}$ and they are disjoint up to end point. Then we have {a} path between $a$, $ b $
which doesn't use $\{w,u\}$.

\item Let $a\in K(v,G),b\in V(G)\setminus K(v,G)$, connect them with $l
\quad$ \\ $ l=a_1,a_2,\dots, a_p=a $ and
$l=b_1,b_2,\dots ,b_q=b \quad$,
$\tilde{a}=a_{max\{i|a_i=b_i\}}$. {The} path between $a$,$l$
{doesn't} use $\{w,u\}$ and {the} path between $b$,$l$
{uses} $\{w,u\}$. Thus the path between $\tilde{a}$,$a$ {uses} that edge and
{the} path {between} $\tilde{a}$,$b$ 
doesn't use $\{w,u\}$ and they are disjoint up to end point.
Then we have {a} path between $a$, $ b $
which {uses} $\{w,u\}$.

\end{enumerate}

\end{proof}

\begin{definition}

Define graphs \[
\bar{G}=(V(G),E(G)\setminus \{\{w,u\}\} \cup \{\{w,l\}\}), \]
\[ \tilde{G}=(V(G),E(G)\setminus \{\{w,u\}\} \cup \{\{w,k\}\}).
\]
\end{definition}

\begin{lemma}\label{l9}

$\bar{G}$ {and} $\tilde{G}$ are trees.
\end{lemma}

\begin{proof}

At first we show connectivity, let $a,b \in V(G)$, connect them with a path. If both are in $K(v,G)$
or in $V(G)
\setminus K(v,G)$, then by lemma \ref{l8}, the path in $G$ {uses} only {edges} which
are also in $\bar{G},\tilde{G}$. {Hence} it is path also there.

{Let} $a \in K(v,G)$ {and} $b \in V(G)
\setminus K(v,G)$. We can see $w\in K(v,G)$, by lemma \ref{l8} {a} path
between $a$ and
$w, \quad a=a_1,a_2, \dots ,a_p=w$, {doesn't} use $\{w,u\}$ and all vertices of
this path are in $K(v,G)$. If {not}, there is {a} path between vertices from $K(v,G)$
and $V(G)
\setminus K(v,G)$ which {doesn't} use $\{w,u\}$, that is {a} contradiction with
lemma \ref{l8}.
{Connect $l$ and $b$ with a path},$\quad l=b_1,b_2, \dots ,b_q=b$. It {doesn't} use $\{w,u\}$  
and all vertices {are} in $V(G)\setminus
K(v,G)$. Then $a=a_1,a_2, \dots
,a_p=w,l=b_1,b_2, \dots ,b_q=b$ is a path between $a,b$ in {the} graph $\bar{G}$,
analogically for $\tilde{G}$.

Now we show {that} they don't contain {a} cycle, for contradiction {suppose that} $\bar{G}$ contains {a}
cycle $K\subseteq \bar{G}$. If $K$ {doesn't} use {the} edge $\{w,l\}$, {then} 
$K\subseteq G$, but $G$ is {a} tree, {this is a contradiction}. If $K$ {uses}
$\{w,l\}$, {then there exists a path in $G$ between $w,l$}, which {doesn't} use {the}  
edge $\{w,l\}$. {Then there exists a path in $G$ between $w,l$}, which {doesn't} use {the}  
edge $\{w,u\}$, but $w\in K(v,G)$ and $l\in V(G) \setminus K(v,G)$, that is
contradiction with lemma \ref{l8}. Analogically for $\tilde{G}$.

\end{proof}

We want to show that 
\[
s_H(G,f)\leq s_H(\bar{G},f)
\]
or
\[
s_H(G,f)\leq s_H(\tilde{G},f)
\]

\begin{lemma}\label{l10}

\[a,b\in K(v,G),\phantom{\setminus V(G)} \quad\mathrm{then}
\quad\rho_{G}(a,b)=\rho_{\bar{G}}(a,b)=\rho_{\tilde{G}}(a,b),\]

\[a,b\in V(G)\setminus K(v,G), \quad\mathrm{then}
\quad \rho_{G}(a,b)=\rho_{\bar{G}}(a,b)=\rho_{\tilde{G}}(a,b).\]

\end{lemma}

\begin{proof}

A path in $G$ between $a,b$, by lemma \ref{l8}, doesn't use $\{u,w\}$ , {hence} it is {a} path in $\bar{G}$ and $\tilde{G}$ too, then {the} distance of $a,b$ is {the} same in $G, \bar{G}$ and $\tilde{G}$.
\end{proof}

\begin{definition}

Define subsets \[ 
F^{+},F^{-},F^{0}\subseteq K(v,G) \times (V(G)\setminus K(v,G))
\]
{such that $(a,b) \in F^{+}$ if a path between $a,b$ uses the edge
$\{u,u^{+}\}$.}
$(a,b) \in F^{-}$ if a path between $a,b$ {uses} the edge $\{u,u^{-}\}$
and $(a,b) \in F^{0}$ if a path between $a,b$ doesn't use {neither $\{u,u^{-}\}$ nor
$\{u,u^{+}\}$}.

\end{definition}

\begin{lemma}\label{l11}

$F^+,F^-,F^0$ are pairwise disjoint and \[ F^+\cup F^- \cup F^0= K(v,G) \times (V(G)\setminus
K(v,G)). \]

\end{lemma}

\begin{proof}

From {the} definition of $F^+,F^-,F^0$ {we} have $F^-$ and $F^0$, $F^+$ and $F^0$ are disjoint.
Let $(a,b)\in F^+ \cap F^-$, then the path between $a,b$ {uses edges}
$\{u,u^-\},\{u,u^+\}$ and by lemma~\ref{l8}, it also {uses} the edge $\{w,u\}$. Hence it is a path which {has} a 
vertex of degree $3$ and that is contradiction.

\end{proof}

\begin{lemma}\label{l12}
 
Let $x,\bar{x} \in K(v,G)$ and $y,\bar{y} \in
V(G)\setminus K(v,G)$ {such} that $(x,y) \in F^+$ and $(\bar{x},\bar{y})\in F^-$.
{Then}
 \[\rho_{\bar{G}}(x,y)+\rho_{\bar{G}}(\bar{x},\bar{y})\geq
\rho_{G}(x,y)+\rho_{G}(\bar{x},\bar{y}), 
\]
\[\rho_{\tilde{G}}(x,y)+\rho_{\tilde{G}}(\bar{x},\bar{y})\geq \rho_{G}(x,y)+\rho_{G}(\bar{x},\bar{y}).
\]
Moreover, {both} sides are equal, in {the} first inequality, if and only if $y=l$ and{, in the second inequality}, if and only if $\bar{y} =k$.

\end{lemma}

\begin{proof}



Let $z$ denote the first common vertex of paths $Q: l=y_1,y_2,\dots
,y_s=k$ and $P:y=x_1,x_2,\dots ,x_m=x$.
Consider 
\[
i_{m}=min\{i|\exists j \in \{1,\dots ,m \},y_i=x_j\}
\]
and therefore $z=y_{i_m}$, let $T$ be the path from $z$ to $l$, we will {show that} $z$ is 
the only one common vertex of $T$ and $P$, vertices from $P$ split into the 4 subpaths, $P_1$ from $y$ to $z$,
$P_2$ from $z$ to $u$, edge $\{u,w\}$ and $P_3$ from $w$ to $x$. Vertices from $P_1$
are not in $Q$ (except {for} $z$) {from the} definition of $z$. 
Vertices from $P_2$ are not in $T$ (except {for} $z$) from {the} uniqueness of {paths} in trees and vertices from $P_3$ {belong} to 
$K(v,G)$ and every {vertex} of $T$ belongs to $V(G)\setminus K(v,G)$.   
By composition of paths $P_1,T,\{l,w\},P_3$, we get {a} path from $y$ to $x$ in {the} 
graph $\bar{G}$.

Let $\bar{P}$ denote the path from $\bar{y}$ to $\bar{x}$, analogically define
$\bar{z}$ {as the first} common vertex of paths $\bar{P}$ and $Q$ (first in {the} direction from $\bar{y} $ to
$\bar{x}$ ). We split $\bar{P}$ into the subpaths $\bar{P}_1$ from
$\bar{y}$ to $\bar{z}$ , $\bar{P}_2$ from $\bar{z}$ to $u$, edge $\{u,w\}$ and
$\bar{P}_3$ from $u$ to $\bar{x}$. Let $\bar{T}$ be the  path from $u$ to $l$, analogically we get {that} $u$ is the only one common vertex of 
$\bar{P}$ and $\bar{T}$. Hence 
$\bar{P}_1,\bar{P}_2,\bar{T},\{l,w\},\bar{P}_3$ is {a} path between
$\bar{y},\bar{x}$ in {the} graph $\bar{G}$.

And for paths from $u$ to $z$ and from $u$ to $\bar{z}$, $u$ is the only one common vertex, by uniqueness of path in trees. 

Now we can calculate.
\begin{equation}
\begin{aligned}
\rho_G(x,y)&=\rho_G(x,w)+1+\rho_G(u,z)+\rho_G(z,y),
\\
\rho_G(\bar{x},\bar{y})&=\rho_G(\bar{x},w)+1+\rho_G(u,\bar{z})+\rho_G(\bar{z},\bar{y}),
\\
\rho_{\bar{G}}(x,y)&=\rho_G(x,w)+1+\rho_G(l,z)+\rho_G(z,y),
\\
\rho_{\bar{G}}(\bar{x},\bar{y})&=\rho_G(\bar{x},w)+1+\rho_G(l,z)+\rho_G(z,u) +
\rho_G(u,\bar{z})+ \rho_G(\bar{z},\bar{y}),\\
\\
\text{hence}
\\
\\
\rho_{\bar{G}}(\bar{x},\bar{y})+\rho_{\bar{G}}(x,y) &=
\rho_G(\bar{x},\bar{y})+\rho_G(x,y)+ 2\rho_G(l,z).
\end{aligned}\notag
\end{equation}
Now we get our inequality and we see {that both} are equal {if} and only if
$l=z$. But $l$ is a leaf, {hence} $z$ is a leaf, then 
$y=z=l$. For $\tilde{G}$ analogically.

\end{proof}

\begin{example}

Paths between $x,y$ and 
$\bar{x},\bar{y}$ in graphs $G$ and $\bar{G}$. 
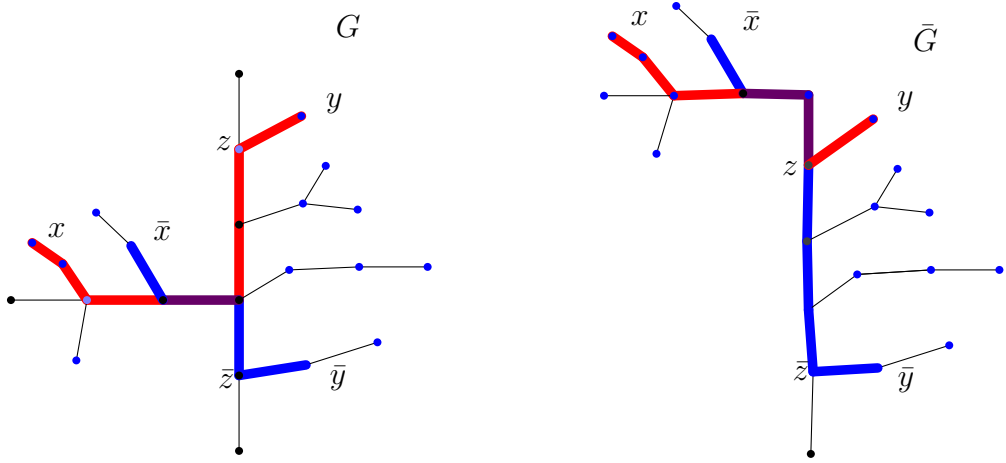
\begin{center}
\definecolor{uququq}{rgb}{0.25,0.25,0.25}
\definecolor{wwqqww}{rgb}{0.4,0,0.4}
\definecolor{qqqqff}{rgb}{0,0,1}
\definecolor{ffqqqq}{rgb}{1,0,0}
\definecolor{xdxdff}{rgb}{0.49,0.49,1}
\begin{tikzpicture}[line cap=round,line join=round,>=triangle 45,x=1.0cm,y=1.0cm]
\clip(-3.25,-4.66) rectangle (10.53,5.55);
\draw (0,3)-- (0,2);
\draw [line width=3.6pt,color=ffqqqq] (0,2)-- (0,1);
\draw [line width=3.6pt,color=ffqqqq] (0,1)-- (0,0);
\draw [line width=3.6pt,color=qqqqff] (0,0)-- (0,-1);
\draw (0,-1)-- (0,-2);
\draw [line width=3.6pt,color=wwqqww] (0,0)-- (-1,0);
\draw [line width=3.6pt,color=ffqqqq] (-1,0)-- (-2,0);
\draw (-2,0)-- (-3,0);
\draw [line width=3.6pt,color=ffqqqq] (-2,0)-- (-2.32,0.48);
\draw [line width=3.6pt,color=ffqqqq] (-2.32,0.48)-- (-2.72,0.76);
\draw [line width=3.6pt,color=qqqqff] (-1,0)-- (-1.42,0.72);
\draw (-1.42,0.72)-- (-1.88,1.16);
\draw (0,1)-- (0.84,1.28);
\draw (0.84,1.28)-- (1.56,1.2);
\draw [line width=3.6pt,color=qqqqff] (0,-1)-- (0.88,-0.86);
\draw (0.88,-0.86)-- (1.82,-0.56);
\draw (0,0)-- (0.66,0.4);
\draw (0.66,0.4)-- (1.58,0.44);
\draw (0.84,1.28)-- (1.14,1.78);
\draw [line width=3.6pt,color=ffqqqq] (0,2)-- (0.82,2.44);
\draw (1.58,0.44)-- (2.48,0.44);
\draw (-2,0)-- (-2.14,-0.8);
\draw [line width=3.6pt,color=ffqqqq] (5.31,3.22)-- (4.91,3.5);
\draw [line width=3.6pt,color=qqqqff] (6.63,2.74)-- (6.21,3.46);
\draw (6.21,3.46)-- (5.75,3.9);
\draw (8.36,1.24)-- (9.08,1.16);
\draw (8.4,-0.9)-- (9.34,-0.6);
\draw (8.13,0.34)-- (9.1,0.4);
\draw (8.36,1.24)-- (8.66,1.74);
\draw (9.1,0.4)-- (10,0.4);
\draw (7.52,-2.04)-- (7.55,-0.95);
\draw [line width=3.6pt,color=qqqqff] (7.55,-0.95)-- (7.49,-0.13);
\draw [line width=3.6pt,color=qqqqff] (7.49,-0.13)-- (7.47,0.78);
\draw [line width=3.6pt,color=qqqqff] (7.47,0.78)-- (7.49,1.79);
\draw [line width=3.6pt,color=wwqqww] (7.49,1.79)-- (7.49,2.72);
\draw (9.1,0.4)-- (8.13,0.34);
\draw (7.49,-0.13)-- (8.13,0.34);
\draw (8.36,1.24)-- (7.47,0.78);
\draw [line width=3.6pt,color=ffqqqq] (8.34,2.4)-- (7.49,1.79);
\draw [line width=3.6pt,color=qqqqff] (8.4,-0.9)-- (7.55,-0.95);
\draw [line width=3.6pt,color=ffqqqq] (6.63,2.74)-- (5.72,2.71);
\draw (5.72,2.71)-- (5.49,1.94);
\draw [line width=3.6pt,color=ffqqqq] (5.72,2.71)-- (5.31,3.22);
\draw (5.72,2.71)-- (4.8,2.71);
\draw [line width=3.6pt,color=wwqqww] (6.63,2.74)-- (7.49,2.72);
\draw (-2.66,1.16) node[anchor=north west] {$x$};
\draw (-1.27,1.18) node[anchor=north west] {$\bar{x}$};
\draw (1,2.88) node[anchor=north west] {$y$};
\draw (1.05,-0.73) node[anchor=north west] {$\bar{y}$};
\draw (-0.45,2.32) node[anchor=north west] {$z$};
\draw (-0.42,-0.78) node[anchor=north west] {$\bar{z}$};
\draw (5.01,3.98) node[anchor=north west] {$x$};
\draw (6.48,3.97) node[anchor=north west] {$\bar{x}$};
\draw (7,1.99) node[anchor=north west] {$z$};
\draw (8.52,-0.76) node[anchor=north west] {$\bar{y}$};
\draw (7.14,-0.68) node[anchor=north west] {$\bar{z}$};
\draw (8.51,2.91) node[anchor=north west] {$y$};
\draw (1.13,3.92) node[anchor=north west] {$G$};
\draw (8.72,3.79) node[anchor=north west] {$\bar{G}$};
\begin{scriptsize}
\fill [color=black] (0,3) circle (1.5pt);
\fill [color=xdxdff] (0,2) circle (1.5pt);
\fill [color=black] (0,1) circle (1.5pt);
\fill [color=black] (0,0) circle (1.5pt);
\fill [color=black] (0,-2) circle (1.5pt);
\fill [color=black] (-1,0) circle (1.5pt);
\fill [color=xdxdff] (-2,0) circle (1.5pt);
\fill [color=black] (-3,0) circle (1.5pt);
\fill [color=qqqqff] (-2.32,0.48) circle (1.5pt);
\fill [color=qqqqff] (-2.72,0.76) circle (1.5pt);
\fill [color=qqqqff] (-1.42,0.72) circle (1.5pt);
\fill [color=qqqqff] (-1.88,1.16) circle (1.5pt);
\fill [color=qqqqff] (0.84,1.28) circle (1.5pt);
\fill [color=qqqqff] (1.56,1.2) circle (1.5pt);
\fill [color=qqqqff] (0.88,-0.86) circle (1.5pt);
\fill [color=qqqqff] (1.82,-0.56) circle (1.5pt);
\fill [color=qqqqff] (0.66,0.4) circle (1.5pt);
\fill [color=qqqqff] (1.58,0.44) circle (1.5pt);
\fill [color=qqqqff] (1.14,1.78) circle (1.5pt);
\fill [color=qqqqff] (0.82,2.44) circle (1.5pt);
\fill [color=qqqqff] (2.48,0.44) circle (1.5pt);
\fill [color=qqqqff] (-2.14,-0.8) circle (1.5pt);
\fill [color=black] (0,-1) circle (1.5pt);
\fill [color=black] (7.52,-2.04) circle (1.5pt);
\fill [color=black] (6.63,2.74) circle (1.5pt);
\fill [color=qqqqff] (5.31,3.22) circle (1.5pt);
\fill [color=qqqqff] (4.91,3.5) circle (1.5pt);
\fill [color=qqqqff] (6.21,3.46) circle (1.5pt);
\fill [color=qqqqff] (5.75,3.9) circle (1.5pt);
\fill [color=qqqqff] (8.36,1.24) circle (1.5pt);
\fill [color=qqqqff] (9.08,1.16) circle (1.5pt);
\fill [color=qqqqff] (8.4,-0.9) circle (1.5pt);
\fill [color=qqqqff] (9.34,-0.6) circle (1.5pt);
\fill [color=qqqqff] (8.13,0.34) circle (1.5pt);
\fill [color=qqqqff] (9.1,0.4) circle (1.5pt);
\fill [color=qqqqff] (8.66,1.74) circle (1.5pt);
\fill [color=qqqqff] (8.34,2.4) circle (1.5pt);
\fill [color=qqqqff] (10,0.4) circle (1.5pt);
\fill [color=qqqqff] (5.49,1.94) circle (1.5pt);
\fill [color=qqqqff] (7.55,-0.95) circle (1.5pt);
\fill [color=qqqqff] (7.49,-0.13) circle (1.5pt);
\fill [color=qqqqff] (7.49,2.72) circle (1.5pt);
\fill [color=qqqqff] (5.72,2.71) circle (1.5pt);
\fill [color=qqqqff] (4.8,2.71) circle (1.5pt);
\fill [color=uququq] (7.47,0.78) circle (1.5pt);
\fill [color=uququq] (7.49,1.79) circle (1.5pt);
\end{scriptsize}
\end{tikzpicture}
\end{center}

\end{example}

\begin{lemma}\label{l13}

Let $(x,y) \in F^0$ then \[
\rho_{\bar{G}}(x,y)> \rho_G(x,y),
\]
\[
\rho_{\tilde{G}}(x,y)> \rho_G(x,y).
\]
\end{lemma}

\begin{proof}

Let $P$ {be} a path from $x$ to $y$ and $Q$ {be} a path from $l$ to $k$ in $G$,
for $P$ and $Q$, $u$ is the only one common vertex {because} $(x,y) \in F^0$. Hence 
$x\rightarrow w \-- l\rightarrow u \rightarrow y$ is a path in
$\bar{G}$, where paths of type $a\rightarrow b$ are subpaths of $P$ and $Q$ and $\-- $ {denotes an} edge.  
Now we can calculate {the following}.
\[ 
\rho_{\bar{G}}(x,y)=\rho_G(x,u)+1+\rho_G(l,u)+\rho_G(u,y)=\rho_G(x,y)+\rho_G(l,u)
\]
and from $l \neq u$ {we} have inequality.

For $\tilde{G}$ analogically.
\end{proof}

\begin{lemma}\label{l14}

\[\rho_{\bar{G}}(x,y) > \rho_G(x,y)\quad \text{for }(x,y) \in F^-,\]
\[\rho_{\tilde{G}}(x,y) > \rho_G(x,y)\quad \text{for }(x,y) \in F^+.\]

\end{lemma}

\begin{proof}
We will prove {the} first inequality.
As well as in lemma \ref{l12} denote $z$ the first common vertex of paths from $y$ to $x$ and from $k$ to $l$, 
formally we can define {it} as well as in lemma \ref{l12}. Now we consider a path \\ $x \rightarrow w \-- l \rightarrow u \rightarrow z
\rightarrow y$. Hence
\[
\rho_{\bar{G}}(x,y)=\rho_G(x,w)+1+\rho_G(l,u)+\rho_G(u,z)+\rho_G(z,y)=\rho_G(x,y)+\rho_G(l,u)
\]
and from $l \neq u$ {we} have inequality.
 
For second inequality analogically.
\end{proof}

\begin{definition}

Let $G$ and $H$ be graphs such that $G$ is connected, \\ $|V(G)|~=~|V(H)|$ and 
$f:V(H)\rightarrow V(G)$ {is a} pseudoordering, we define a set 

\[
L=\{(x,y)\in K(v,G)\times (V(G)\setminus K(v,G))|x\sim_{H,f} y\}.
\]
\end{definition}

\begin{lemma}\label{l15}

Let $G$ and $H$ be graphs such that $G$ is connected, \\ $|V(G)|~=~|V(H)|$ and 
$f:V(H)\rightarrow V(G)$ {is a} pseudoordering. {Then}
\[
s_H(f,\bar{G}) \geq s_H(f,G),
\]
or
\[
s_H(f,\tilde{G}) \geq s_H(f,G),
\]
the first case occurs when

 $|L \cap F^+| \leq |L \cap F^-|$, the second case occurs when
 $|L \cap F^+|
\geq |L \cap F^-|$.
\end{lemma}

\begin{proof}

Denote $n^+=|L \cap F^+|,\quad n^-=|L \cap F^-|$,
let $n^+ \geq n^-$, the second case is analogical, we rearrange {the} sum $s_H(f,G)$
in this way.
\[
s_H(f,G)=\sum_{i=1}^{n^-} (\rho_G(x_i,y_i)+\rho_G(\bar{x}_i,\bar{y}_i)) +
\sum_{i=n^-+1}^{n^+}\rho_G(x_i,y_i)+
\sum_{i=1}^{m}\rho_G(a_i,b_i)+\sum_{i=1}^{\bar{m}}\rho_G(c_i,d_i), 
\]
where
\[(x_i,y_i) \in F^+, \quad (\bar{x}_i,\bar{y}_i) \in F^-, \quad (a_i,b_i)
\in F^0, \quad (c_i,d_i) \notin L.
\]
Now, by lemma \ref{l12}
\[
\rho_G(x_i,y_i)+\rho_G(\bar{x}_i,\bar{y}_i)\leq
\rho_{\tilde{G}}(x_i,y_i)+\rho_{\tilde{G}}(\bar{x}_i,\bar{y}_i),
\]
by lemma \ref{l14}
\[
\rho_G(x_i,y_i)\leq\rho_{\tilde{G}}(x_i,y_i),
\]
by lemma \ref{l13}
\[
\rho_G(a_i,b_i)\leq\rho_{\tilde{G}}(a_i,b_i)
\]
and by lemma \ref{l10}
\[
\rho_G(c_i,d_i)=\rho_{\tilde{G}}(c_i,d_i).
\]
Hence

\[
\begin{aligned}
&s_H(f,G)\leq \\
&\leq\sum_{i=1}^{n^-} (\rho_{\tilde{G}}(x_i,y_i)+\rho_{\tilde{G}}(\bar{x}_i,\bar{y}_i)) +
\sum_{i=n^-+1}^{n^+}\rho_{\tilde{G}}(x_i,y_i)+\sum_{i=1}^{m}\rho_{\tilde{G}}(a_i,b_i)+\sum_{i=1}^{\bar{m}}\rho_{\tilde{G}}(c_i,d_i)=
\\
&=s_H(f,\tilde{G}).
\end{aligned}
\]

\end{proof}

\begin{lemma}\label{l16}
Let $G$ and $H$ be graphs such that $G$ is connected, $|V(G)|~=~|V(H)|$ and 
$f:V(H)\rightarrow V(G)$ {is a} pseudoordering. {Then there} exists {a pseudoordering}
\\
$g:V(H) \rightarrow \{x_1,x_2,\dots
,x_{|V(G)|}\}~=~V(P_{|V(G)|-1})$ such that 
\[
s_H(f,G) \leq s_H(g,P_{|V(G)|-1}).
\]
\end{lemma}

\begin{proof}
We denote \[
\alpha(G)=\sum_{\substack{v \in V(G) \\ deg_Gv \geq 3}} deg_Gv ,
\]
from {the} definition of $u,l$ and $k$ we know {that} $deg_G u \geq 3$ and $deg_G l=deg_G k=1$.
From {the} construction of $\bar{G}$ and $\tilde{G}$
we have $deg_{\bar{G}} u=deg_{\tilde{G}} u \leq deg_G u$,  
$deg_{\bar{G}} l=deg_{\tilde{G}} k=2$ and all other vertices have the same
degree as before. {Hence} \[
\alpha(\bar{G})<\alpha(G),
\]
\[
\alpha(\tilde{G})<\alpha(G).
\]
Let $S$ be a tree, which is not a path, we choose any three {pairwise distinct leaves} in
$V(S)$ and define $S^*$ {as} one of graphs
$\bar{S},\tilde{S}$, which satisfy
$s_H(f,S^*)~\geq~s_H(f,S)$.
Denote $G_0=G$ and for $i\geq 0$ denote $G_{i+1}=G_i^*$ {if} $G_i$ is not a path,
otherwise define $G_{i+1}=G_i$.
For contradiction we assume {that the} tree $G_i$ is not a path for every $i \in \mathbb{N}_0$.
We know $\alpha(G_i) \in  \mathbb{N}_0$ for every $i$ and 
\[
\alpha(G_{i+1})\leq \alpha(G_i)-1,
\]
hence
\[
\alpha(G_{\alpha(G_0)+1}) \leq \alpha(G_O)-\alpha(G_0)-1=-1
\]
and this is contradiction. Therefore {there} exists some $j$ such that $G_j$ is a path, from lemma \ref{l15} we get  
\[
s_H(f,G_{i+1})\geq s_H(f,G_i)
\]
and {hence}
\[
s_H(f,G_j) \geq s_H(f,G).
\]

\end{proof}

\begin{theorem}\label{v1}
Let $G$ and $H$ be graphs such that $G$ is connected, \\ $|V(G)|~=~|V(H)|$ and 
$f:V(H)\rightarrow V(G)$ {is a} pseudoordering, then {there} exists {a pseudordering}
\\
$g:V(H) \rightarrow \{x_1,x_2,\dots
,x_{|V(G)|}\}~=~V(P_{|V(G)|-1})$ such that  
\[
s_H(f,G) \leq s_H(g,P_{|V(G)|-1}).
\]

\end{theorem}

\begin{proof}
Let $K$ {be} any {spanning} tree of $G$, $x,y\in V(G)$, we connect $x$ and $y$ with a path 
in graph $K$, this path is also {a} path in $G$. Hence 
\[
\rho_G(x,y) \leq \rho_K(x,y)
\] 
{for} every $x,y$, hence 
\[
s_H(f,G)\leq s_H(f,K),
\]
by lemma \ref{l16} {there} exists {a pseudoordering} \\
$g:V(H) \rightarrow \{x_1,x_2,\dots
,x_{|V(G)|}\}~=~V(P_{|V(G)|-1})$ such that
\[
s_H(f,G) \leq s_H(f,K) \leq s_H(g,P_{|V(G)|-1}).
\]

\end{proof}

\begin{corollary}\label{do1}
Let $G$ and $H$ be graphs such that $G$ is connected, \\ $|V(G)|~=~|V(H)|$, {then}
\[
h^+_{H}(G)\leq h^+_{H}(P_{|V(G)|-1}).
\]

\end{corollary}

\section{Graphs with a maximal upper H-Hamiltonian number}

In this section we will prove {that} if in corollary \ref{do1} {the} graph $H$ is connected,
{then} in {the} inequality in corollary \ref{do1} {both} sides are equal.

\begin{remark}\label{pozn1}
For easier writing, we will {denote} vertices of $H$ {the} same as vertices of
$G$, we will rename them in this way $v\in H \mapsto f(v)$. We can naturally see it as graph with two sets of edges.

In inequalities in lemma \ref{l15} {both} sides are equal under specific conditions,
if $L\cap F^0\neq \emptyset$, {then} in lemma \ref{l13} {there is a} strict inequality and {then} also {the same happens} in theorem
\ref{v1}.

If $(L\setminus K(v,G)\times \{l\})\cap F^+\neq \emptyset$, {then} in lemma
\ref{l12} {there is a} strict inequality and {then} also {the same happens} in theorem \ref{v1}.
{Analogically} if $(L\setminus K(v,G)\times \{k\})\cap F^-\neq \emptyset$.

Overall we get {that} the only nontrivial case is 
\begin{equation}\label{r3}
L\subseteq K(v,G)\times \{k,l\}.
\end{equation}

\end{remark}

\begin{remark}
At the beginning we took three arbitrary leafs $k,l,v$ and we got remark
\ref{pozn1}, now we take another $\bar{k},\bar{l},\bar{v}$ for them we {also have} remark \ref{pozn1}.
\end{remark}

\begin{lemma}\label{l18}
Let $G$ be a tree, $H$ connected graph such that  
$|V(G)|=|V(H)|$ and $f:V(H)\rightarrow V(G)$ {is a} pseudoordering, which satisfy
\[
s_H(f,G)= h^+_{H}(P_{|V(G)|-1}),
\]

{then} $G$ is path.

\end{lemma}

\begin{proof}

For contradiction suppose {that} $G$ is not a path, then there {exist} three {pairwise distinct} {leaves} $k,l,v${, we} denote {in the same way} as before, vertex $u$ and set 
of vertices $K(v,G)$. Because graph $H$ is connected {there} exists {a} vertex $x$ such that $\{u,x\}\in E(H)$.
Let $X\subseteq V(G)$ be {a} set of vertices of {components} of graph $G\setminus u$, graph $G$ if we delete vertex $u$, containing $x$.
$G\setminus u$ {has}, by definition of $u${, at} least $3$ components. Let now $\bar{v}$ be an arbitrary leaf (leaf {in $G$}) in $X$. 
Choose $\bar{k},\bar{l}$ as arbitrary {leaves} in {pairwise distinct components} of $G\setminus u$ and different from $X$.

Now $(x,u)\in \bar{L}$, where $\bar{L}$ is alternative of $L$ for $\bar{k},\bar{l},\bar{v}$ and by remark 
\ref{pozn1} for $\bar{k},\bar{l},\bar{v}$ and by $k\neq u\neq l$ we get contradiction.

\end{proof}

\begin{example}
We show {the} idea of {the} last proof {in the following} picture.

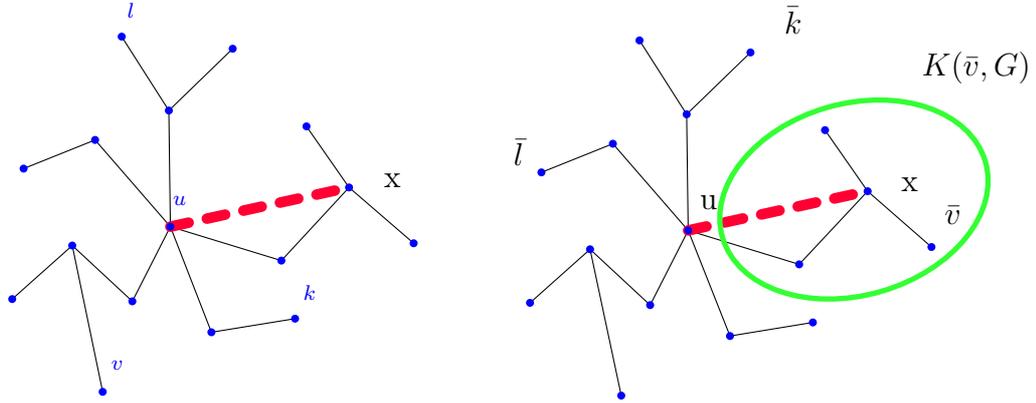
\begin{center}

\definecolor{ttfftt}{rgb}{0.2,1,0.2}
\definecolor{ffqqtt}{rgb}{1,0,0.2}
\definecolor{qqqqff}{rgb}{0,0,1}
\begin{tikzpicture}[line cap=round,line join=round,>=triangle 45,x=1.0cm,y=1.0cm]
\clip(-2.75,-6.13) rectangle (13.18,6.65);
\draw (-0.71,2.75)-- (0.28,1.6);
\draw (0.28,1.6)-- (0.26,3.14);
\draw (0.26,3.14)-- (-0.36,4.12);
\draw (0.26,3.14)-- (1.1,3.96);
\draw (0.28,1.6)-- (1.74,1.15);
\draw (1.74,1.15)-- (2.63,2.12);
\draw (2.63,2.12)-- (2.07,2.93);
\draw (2.63,2.12)-- (3.48,1.38);
\draw (0.28,1.6)-- (0.82,0.2);
\draw (0.82,0.2)-- (1.92,0.38);
\draw (0.28,1.6)-- (-0.22,0.61);
\draw (-0.22,0.61)-- (-1.01,1.35);
\draw (-1.8,0.64)-- (-1.01,1.35);
\draw (-1.01,1.35)-- (-0.61,-0.59);
\draw (-0.71,2.75)-- (-1.65,2.37);
\draw [line width=4pt,dash pattern=on 7pt off 7pt,color=ffqqtt] (0.28,1.6)-- (2.63,2.12);
\draw (2.95,2.45) node[anchor=north west] {x};
\draw (6.1,2.7)-- (7.09,1.55);
\draw (7.09,1.55)-- (7.07,3.09);
\draw (7.07,3.09)-- (6.45,4.07);
\draw (7.07,3.09)-- (7.91,3.91);
\draw (7.09,1.55)-- (8.55,1.1);
\draw (8.55,1.1)-- (9.45,2.07);
\draw (9.45,2.07)-- (8.89,2.88);
\draw (9.45,2.07)-- (10.29,1.33);
\draw (7.09,1.55)-- (7.64,0.15);
\draw (7.64,0.15)-- (8.73,0.33);
\draw (7.09,1.55)-- (6.59,0.56);
\draw (6.59,0.56)-- (5.8,1.3);
\draw (5.01,0.59)-- (5.8,1.3);
\draw (5.8,1.3)-- (6.21,-0.64);
\draw (6.1,2.7)-- (5.16,2.32);
\draw [line width=4pt,dash pattern=on 7pt off 7pt,color=ffqqtt] (7.09,1.55)-- (9.45,2.07);
\draw (9.75,2.4) node[anchor=north west] {x};
\draw (7.10,2.16) node[anchor=north west] {u};
\draw (10.32,2.02) node[anchor=north west] {$ \bar{v} $};
\draw (4.65,2.91) node[anchor=north west] {$ \bar{l} $};
\draw (8.21,4.66) node[anchor=north west] {$ \bar{k} $};
\draw [rotate around={16.26:(9.27,1.96)},line width=2pt,color=ttfftt] (9.27,1.96) ellipse (1.8cm and 1.27cm);
\draw (10.02,4.07) node[anchor=north west] {$ K(\bar{v},G) $};
\begin{scriptsize}
\fill [color=qqqqff] (-0.71,2.75) circle (1.5pt);
\fill [color=qqqqff] (0.28,1.6) circle (1.5pt);
\draw[color=qqqqff] (0.40,1.94) node {$u$};
\fill [color=qqqqff] (0.26,3.14) circle (1.5pt);
\fill [color=qqqqff] (-0.36,4.12) circle (1.5pt);
\draw[color=qqqqff] (-0.24,4.47) node {$l$};
\fill [color=qqqqff] (1.1,3.96) circle (1.5pt);
\fill [color=qqqqff] (1.74,1.15) circle (1.5pt);
\fill [color=qqqqff] (2.63,2.12) circle (1.5pt);
\fill [color=qqqqff] (2.07,2.93) circle (1.5pt);
\fill [color=qqqqff] (3.48,1.38) circle (1.5pt);
\fill [color=qqqqff] (0.82,0.2) circle (1.5pt);
\fill [color=qqqqff] (1.92,0.38) circle (1.5pt);
\draw[color=qqqqff] (2.11,0.73) node {$k$};
\fill [color=qqqqff] (-0.22,0.61) circle (1.5pt);
\fill [color=qqqqff] (-1.01,1.35) circle (1.5pt);
\fill [color=qqqqff] (-1.8,0.64) circle (1.5pt);
\fill [color=qqqqff] (-0.61,-0.59) circle (1.5pt);
\draw[color=qqqqff] (-0.42,-0.24) node {$v$};
\fill [color=qqqqff] (-1.65,2.37) circle (1.5pt);
\fill [color=qqqqff] (6.1,2.7) circle (1.5pt);
\fill [color=qqqqff] (7.09,1.55) circle (1.5pt);
\fill [color=qqqqff] (7.07,3.09) circle (1.5pt);
\fill [color=qqqqff] (6.45,4.07) circle (1.5pt);
\fill [color=qqqqff] (7.91,3.91) circle (1.5pt);
\fill [color=qqqqff] (8.55,1.1) circle (1.5pt);
\fill [color=qqqqff] (9.45,2.07) circle (1.5pt);
\fill [color=qqqqff] (8.89,2.88) circle (1.5pt);
\fill [color=qqqqff] (10.29,1.33) circle (1.5pt);
\fill [color=qqqqff] (7.64,0.15) circle (1.5pt);
\fill [color=qqqqff] (8.73,0.33) circle (1.5pt);
\fill [color=qqqqff] (6.59,0.56) circle (1.5pt);
\fill [color=qqqqff] (5.8,1.3) circle (1.5pt);
\fill [color=qqqqff] (5.01,0.59) circle (1.5pt);
\fill [color=qqqqff] (6.21,-0.64) circle (1.5pt);
\fill [color=qqqqff] (5.16,2.32) circle (1.5pt);
\end{scriptsize}
\end{tikzpicture}
\end{center}

\end{example}

\begin{remark}
Let $G$ be a graph with {a} maximal $H$-Hamiltonian number, {then} every spanning tree of $G$ 
{has a} maximal $H$-Hamiltonian number, therefore every spanning tree is a path.
We will show {that the} only graphs with this property are cycles and paths.
\end{remark}

\begin{lemma}\label{l19}
Let $G$ be {a} connected graph such that $|V(G)|\geq 2$, {then} there is a vertex, which is not an articulation point.

\end{lemma}

\begin{proof}
Consider a block-cut tree of $G$ and {a} block $B$, which is {a} leaf of {the} block-cut tree or
if this tree {has} only one vertex, then $B=G$. $B$ is{, by definition 
of a block,} 2-connected. {Because $B$ is leaf we get that in $B$ there is only one} articulation and
in $B$ {there} are at least 2 vertices. Hence in $B$ {there} is at least one vertex,
which is not an articulation point.
\end{proof}

\begin{lemma}\label{l20}
Let $G$ {be a} finite connected graph such that $|V(G)|\geq 2$ and every spanning tree of $G$ is a path, {then} 
$G$ is {a} path or {a} cycle.

\end{lemma}

\begin{proof}
We will prove it by induction {with respect to the} number of vertices. Let $n$ be {the} number of
vertices, for $n=2$ and $n=3$ it is obviously true. Let it {be} true for $n\geq3$, let $G$ be a graph with $n+1$ vertices such that every spanning tree of $G$ is a path. Let $v \in V(G)$ {be a} vertex, which is not an articulation point,
by lemma \ref{l19} it exists. We denote $G'$ {the} subgraph induced by {the} set of vertices
$V(G)\setminus \{v\}$. $G'$ is connected, we will show {that} every spanning tree of $G'$
is a path. Let {there exist a} spanning tree which is not a path, let $u \in V(G)$ {be a} vertex such that $\{v,u\}\in
E(G)$. Now when we add this edge to {the} spanning tree, we get a spanning tree of $G$,
which is not a path and it is a contradiction. 
By induction hypothesis $G'$ is {a} path or {a} cycle,
we denote $A=\left\{u\in V(G)|\{v,u\}\in E(G)\right\}$. For contradiction we assume $G'$ is {a} cycle and let $u \in A$, in $G'$ 
{be an} edge $e$ such that $u$ is not incident to $e$. Consider {the subgraph $B$ of} $G$, $B=(V(G),E(G')\setminus e \cup \{v,u\})$, and this is {a} 
spanning tree of $G$ which is not a path, contradiction.

Therefore $G'$ is {a} path, let $x$, $y$ {be} endpoints of this path, for contradiction we assume {that there} 
exists some another vertex $u \in A$. Hence $G'$ together with $\{u,v\}$ form a spanning tree which is not a path.
Hence $A\subset \{x,y\} $ and
$A\neq \emptyset$ and that are {the} two cases for $G$, {a path and a cycle}.
\end{proof}

\begin{theorem}\label{v2}

Let $G$ and $H$ {be} connected finite graphs such that
$|V(G)|=|V(H)|$, {then}
\[
h^+_H(G)\leq h^+_{H}(P_{|V(G)|-1}),
\]

moreover, {both sides} are equal {if} and only if $G$ is a path.

\end{theorem}

\begin{proof}
{The first part} follows from theorem \ref{v1}, let $G$ be a graph, $f$ {be a} pseudoordering such that
\[
s_H(f,G)=h^+_H(G)=h^+_{H}(P_{|V(G)|-1}).
\]
From {the} proof of theorem \ref{v1} we know {that} every spanning tree also {satisfies the} equation above.
Hence{, by}
lemma \ref{l18}, every spanning tree of $G$ is a path. By lemma \ref{l20} $G$ {is a} path or {a} cycle, 
for contradiction we assume, that it is a cycle. We denote $n=|V(G)|$, we will {show that there} are two vertices $v,u \in V(G)$
such that $v\sim_{H,f}u$
and $\rho_G(u,v)<\frac{n}{2}$. 

{Because} $G$ is cycle,
$|V(H)|=n\geq3$ and $H$ is connected {we see that} there is {a} vertex of degree at least $2$. 
Let $v$ {be a} vertex such that $deg_H(v)\geq 2$, there exists at least two vertices 
$u$ such that 
$v \sim_{H,f} u$. {There exists at most one vertex such that $\rho_G(u,v)\geq \frac{n}{2}$}, {hence} at least one of them {satisfies} $\rho_G(u,v)<
\frac{n}{2}$.

Now we connect $v$ and $u$ with {a} shorter path in $G$. Let $e$ {be} some edge on this path,
we define {a} graph 
$\bar{G}=(V(G),E(G) \setminus e)$, it is {a} path, where {every distance is greater or equal} as in $G$.
But
$\rho_G(u,v) <\rho_{\bar{G}}(u,v)$ and then
\[
s_H(f,\bar{G})=s_{H}(f,\bar{G}) > h^+_{H}(P_{|V(G)|-1}),
\]
and this is contradiction with theorem \ref{v1}.
\end{proof}

\section{Conclusion}

When we use {the} calculation from article \cite{2}, where {it is shown that} 
 
\[
h^+(P_{|V(G)|-1})=\left \lfloor\frac{|V(G)|^2}{2} \right \rfloor, \quad t^+(P_{|V(G)|-1})=\left \lfloor\frac{|V(G)|^2}{2} \right
\rfloor-1.
\]
This result is also calculated in \cite{bak} and when we use theorem \ref{v2} for $H=P_{|V(G)|-1}$ and for $H=C_{|V(G)|}$ we get this theorem
from article \cite{2}.

\begin{theorem}\label{v3}
\[
h^+(G) \leq \left \lfloor\frac{|V(G)|^2}{2} \right \rfloor, \quad t^+(G) \leq \left \lfloor\frac{|V(G)|^2}{2} \right
\rfloor-1.
\]
{Moreover}, {both sides} are equal {if} and only if $G$ is a path.
\end{theorem}
Now we can see, that theorem \ref{v2} is generalization of theorem \ref{v3} from article \cite{2}.

\renewcommand{\bibname}{references}

\end{document}